\documentclass[12pt,reqno]{amsart} 
\usepackage{amsmath,amssymb,amsthm,mathrsfs}
\usepackage[english]{babel}
\usepackage[utf8]{inputenc}
\usepackage{xcolor}

\newtheorem{theorem}{Theorem}[section]
\newtheorem{proposition}[theorem]{Proposition}
\newtheorem{corollary}[theorem]{Corollary}

\theoremstyle{definition}
\newtheorem{definition}[theorem]{Definition}

\newtheorem{remark}[theorem]{Remark}

 \usepackage{multirow} 
\usepackage{color}

\usepackage{graphicx}
\usepackage[left=3cm,right=3cm]{geometry}

\def\r{\mathbb R}

\def\s{\mathbb S}

 \title{The hanging chain problem with respect to a circle}
 \author{Rafael L\'opez}
\address{ Departamento de Geometr\'{\i}a y Topolog\'{\i}a\\  Universidad de Granada. 18071 Granada, Spain}
\email{rcamino@ugr.es}

\begin{document}

\begin{abstract} Let $\s^1$ be a circle in Euclidean plane. We consider the problem of finding the shape of a planar curve which is an extremal of the   potential energy that measures the  distance to $\s^1$.   We describe the shape of these curves distinguishing if the curves lie in the inside or outside  of  $\s^1$. We extend the problem for energies that are powers to the distance to $\s^1$.
  \end{abstract}
  \maketitle

\section{Introduction and motivation of the problem}\label{s1}

The problem of finding a mathematical description of the shape of  hanging chain  acted upon solely by gravity attracted the interest of scientists for long time. Contrary to what Galileo Galilei   thought, the parabola is not the solution of the problem. Separately, Huygens, Leibniz and Johann Bernouilli proved that the solution of the problem is the      catenary,  a much more difficult curve because it involves   exponential functions. Since then, the catenary is of interest not only of mathematicians, but also of physicists, architects and engineers.   The hanging chain problem is a typical introductory problem in many  textbooks of calculus of variations. Also, the catenary is interesting in geometry because it  is the generating curve of the catenoid, the only rotational minimal surface besides the plane. The mathematical literature of the catenary is so extensive that  it would be impossible to include   here. Some of  the most recent works related to catenary  are \cite{bk,cd,dh,kky,mc}. 

From the viewpoint of calculus of variations, the approach to the hanging chain problem is the following. Assume ideal conditions on the chain (constant density, inextensibility) and  constant gravity. Let $(x,y)$ be standard coordinates  of the Euclidean plane $\r^2$. Suppose that  the chain is given by the curve $y=y(x)$, $x\in [a,b]$. The  gravitational energy  of $y=y(x)$ is 
\begin{equation}\label{e0}
\int_{a}^b y(x) \, dl= \int_{a}^b y(x)\sqrt{1+y'(x)^2}\, dx.
\end{equation}
Here $dl=\sqrt{1+y'(x)^2}\, dx$ is the arc element  of the curve   $y(x)$. The integrand in \eqref{e0} represents  the weight of the chain  with respect to the horizontal line $L$ of equation $y=0$ (here density and gravity are equal to $1$). Applying standard techniques of calculus of variations to the energy functional \eqref{e0}, one deduces the differential equation of extremals of this energy. This equation is $yy''=\sqrt{1+y'^2}$ and the solution is   the  catenary whose expression is  $y(x)=a\cosh(\frac{x-b}{a})$, $a,b\in\r$, $a>0$. Recently, the author  has generalized the concept of catenary to other ambient spaces such as the $3$-sphere and the hyperbolic space \cite{lo1}. For this, the reference line $L$  in \eqref{e0} to measure the potential is replaced  by a geodesic of the space. In such a case,  the potential of the curve  is now calculated with the intrinsic distance to that geodesic. In collaboration with L. C. B. Da Silva, the notion of the catenary  have also extended to other ambient spaces, including their relations with minimal surfaces of rotational type: see \cite{dl1,dl2,dl3,lo1,lo2}.

This paper concerns to the  extension of the notion of the catenary considering the original problem in $\r^2$ but replacing the reference line $L$  by   a circle. To be precise, {\it the  hanging problem with respect to a circle} asks what  is the shape of a planar chain whose potential energy is the distance to a given circle.

 Let 
$\s^1=\{(x,y)\in\r^2:x^2+y^2=1\}$ the circle of radius $1$ centered at the origin.  If $(x,y)\in\r^2$,  we use polar coordinates $x=r\cos \theta$, $y=r\sin \theta$ with $r\geq 0$ and $\theta\in\r$. Then the potential of a point  is the Euclidean distance to $\s^1$, that is,  $|r-1|$. Consider an ideal   chain $\gamma\colon [s_1,s_2]\to\r^2$, $I\subset\r$,   suspended from two fixed endpoints. Let us parametrize $\gamma$ in polar coordinates by $\gamma(s)=r(s)(\cos\theta(s),\sin\theta(s))$.   Since the distance of a point $r(s)$ to $\s^1$ is $|r(s)-1|$, the  potential energy of $\gamma$ is
\begin{equation}\label{ee0}
E[\gamma]= \int_{s_1}^{s_2}|r(s)-1|\, dl=\int_{s_1}^{s_2}|r(s)-1|\sqrt{r(s)^2\theta'(s)^2+r'(s)^2}\, ds, 
\end{equation}
because the arc element $dl$   is $\sqrt{r(s)^2\theta'(s)^2+r'(s)^2}\, ds$. The asked shape of the chain appears when the curve reaches a minimum energy value for $E$ for any variation of $\gamma$. As usual in the calculus of variations, tackling this problem is difficult in its generality. For this reason, and as a necessary condition, we only require the curve to be a local extremum of $E$.

\begin{definition} A curve $\gamma\colon I\to\r^2$, $I\subset\r$, which is an extremal (or critical point) of $E$ is said to be   a catenary with respect to the circle $\s^1$.
\end{definition}

If there is no confusion, we will say simply catenary. The  purpose of this paper is describe the catenaries with respect to $\s^1$,   showing their main properties. As we will see,  the solutions cannot intersect the circle $\s^1$. Hence it is a feature of the problem that catenaries will present different behaviors if they lie in each on the  two components of $\r^2\setminus\s^1$. In Sect. \ref{s2} we will see that the catenaries which lie in the inside  of $\s^1$ have periodic functions $r(s)$ (Thm. \ref{t1}).   However, the ones lie in the outside  of $\s^1$ are of curves that are radial graphs on bounded sub-arcs of $\s^1$, being asymptotic to two rays from the origin of $\r^2$ (Thm. \ref{t2}).  In Sect. \ref{s3}, we will extend the initial problem by replacing the distance $|r-1|$ to $\s^1$ by powers of that distance, $|r-1|^\alpha$. 

To conclude this introduction, it is necessary to make the following observation. The classical catenary problem   is not very accurate to real conditions if we consider large scales in the problem. Realistically, on Earth,  the gravitational field is radial and  the distance must measured to the center of the Earth. This was  rightly discussed in \cite{dh}. In that article the potential energy is similar to \eqref{ee}, but instead of $|r-1|$ in the integrand,  it was considered to be $1/r$. Also in the same paper, radial potentials of the form $r^\alpha$ were studied. These energy functionals do not cover the energy \eqref{ee0} because in the hanging chain problem with respect to a circle, the energy is not radial.  


\section{The solution of the hanging problem}\label{s2}

The   extremals  of the energy functional $E$ are obtained by standard arguments of calculus of variations. 

\begin{proposition}\label{pr1}
 Let $\gamma$ be a curve given in polar coordinates $(r,\theta)$. Then  $\gamma$  is an extremal of the energy $E$ given in \eqref{ee} if and only if $\gamma$ is a ray from the origin or $\gamma$ can be parametrized by $\gamma(s)=r(s)(\cos s,\sin s)$ and  $r$ satisfies the second order equation 
\begin{equation}\label{eq2}
r(r-1)r''+r'^2(2-3r)+r^2(1-2r) =0.
\end{equation} 
\end{proposition}

\begin{proof} 
 The Lagrangian of $E$ is 
$$\mathcal{L}(r,r',\theta,\theta')=|r-1|\sqrt{r^2\theta'^2+r'^2}.$$
Using the fundamental lemma of calculus of variations,   the Euler-Lagrange equations  are 
\begin{equation*}
  \frac{\partial \mathcal{L}}{\partial \theta}=\frac{d}{ds}\frac{\partial \mathcal{L}}{\partial \theta'},\quad  \frac{\partial \mathcal{L}}{\partial r}=\frac{d}{ds}\frac{\partial \mathcal{L}}{\partial r'}. 
\end{equation*}
Since $\mathcal{L}$ does not depend on $\theta$, the first equation   implies that there is a constant $c\in\r$ such that 
$$\frac{r(s)^2|r(s)-1|\theta'(s)}{\sqrt{r(s)^2\theta'(s)^2+r'(s)^2}}=c\quad \mbox{for all $s\in I$}.$$
Suppose that $\theta'$ vanishes at some $s=s_1$. Then the above identity implies $c=0$. Hence, $(r(s)-1)\theta'(s)=0$ for all $s\in I$. If $r(s)-1\not=0$ at some $s\in I$, then $\theta'=0$ identically. This proves that   $\theta$ is a constant function and the curve $\gamma$ is a ray from the origin. This proves the first case.  The other case  is that $r(s)=1$ identically. A computation of the  second Euler-Lagrange equation   gives  $\theta'(s)=0$ for all $s\in I$. Thus $\theta$ is a constant function and this would yield that $\gamma$ is a point. This case is not possible. 

Consequently, if $\gamma$ is not a ray from the origin then  the function $\theta'$ cannot vanish. In particular, we can parametrize $\gamma$ by the parameter $\theta$, that is, $\gamma(s)=r(s)(\cos s,\sin s)$. We now compute again the second Euler-Lagrange equation. For this,  we can assume   $r-1\geq 0$   (similarly if $r-1\leq 0$). 
 We calculate each of the terms of this equation, obtaining   
 \begin{equation*}
 \begin{split}
 \frac{\partial \mathcal{L}}{\partial r}&=\frac{2 r^2-  r+r'^2}{\sqrt{r^2+r'^2}}\\
 \frac{d}{ds}\left(\frac{\partial \mathcal{L}}{\partial r'}\right)&=\frac{- r(rr''-r'^2)+r^3r''+r'^4}{(r ^2+r'^2)^{3/2}}.
 \end{split}
 \end{equation*}
 This gives \eqref{eq2}.   
\end{proof}

From now on, the solution given as a ray from the origin will be discarded in this paper. We will also assume that  the catenary is written in nonparametric way $r=r(s)$. In such a case the energy \eqref{ee0} is now 
\begin{equation}\label{ee}
E[\gamma]= \int_{s_1}^{s_2}|r(s)-1|\, dl=\int_{s_1}^{s_2}|r(s)-1|\sqrt{r(s)^2 +r'(s)^2}\, ds, 
\end{equation}

We give another characterization of the catenary in terms of the curvature $\kappa$ of $\gamma$. 

\begin{proposition} \label{pr21}
A curve $\gamma$ is an extremal of $E$ if and only if 
\begin{equation}\label{an}
\kappa=\frac{\cos\varphi}{r-1},
\end{equation}
where $\varphi$ is the angle that makes the unit normal vector $N$ of $\gamma$ with the opposite of the vector $\gamma$.
\end{proposition}

\begin{proof} In polar coordinates, the curvature $\kappa$ is 
$$\kappa=\frac{2r'^2+r^2-rr''}{(r^2+r'^2)^{3/2}}.$$
By substituting into \eqref{eq2}, we obtain 
\begin{equation}\label{k2}
\kappa=\frac{r}{(r-1)\sqrt{r^2+r'^2}}.
\end{equation}
On the other hand, the tangent vector of $\gamma$ is 
$$\gamma'=\frac{(r'\cos s-r\sin s,r'\sin s+r \cos s)}{\sqrt{r^2+r'^2}}.$$
 Thus the unit normal vector of $\gamma$ is 
\begin{equation}\label{nn}
N=\frac{(-r'\sin s-r \cos s,r'\cos s-r\sin s)}{\sqrt{r^2+r'^2}}.
\end{equation}
This gives $\langle N,\gamma\rangle=-\frac{r^2}{\sqrt{r^2+r'^2}}$. Hence, the angle $\varphi$ is given by 
$\cos\varphi=-\frac{r}{\sqrt{r^2+r'^2}}$ because $|\gamma|=r$. This identity and \eqref{k2} characterizes extremals in terms of Eq. \eqref{an}.
\end{proof}

\begin{remark} Identity \eqref{an} extends   the analogous property that has the Euclidean catenary in $\r^2$.   In the classical problem, the Euler-Lagrange equation of the energy \eqref{e0} is $yy''=\sqrt{1+y'^2}$. Since the unit normal     is $N=\frac{(-1,y')}{\sqrt{1+y'^2}}$. Then the Euler-Lagrange equation is equivalent to 
$$\kappa=\frac{1}{y\sqrt{1+y'^2}}=\frac{\cos\psi}{y}.$$
Here $\psi$ is the angle between $N$ and the vector $-(0,1)$, the direction of the gravity. Notice that  $y$ is the distance to the line $L$.
\end{remark}
 
 We now study the solutions of Eq. \eqref{eq2}. By general theory of ODE, and in order to assure existence of a solution $r=r(s)$, $s\in I\subset\r$,  of  \eqref{eq2}, we need to discard the initial condition $r(0)=0$ or $r(0)=1$.  On the other hand, and  by the coefficient $r-1$ for $r''$, it is expectable that the solutions of \eqref{eq2} depend whether the initial condition $r(0)$ is less or bigger than $1$.    
 
Although initial conditions cannot take the values $r=0$ or  $r=1$, it is possible that the solution crosses the origin of $\r^2$ or the circle $\s^1$. We will see that both situations cannot occur. Also, we study if there are constant solutions of \eqref{eq2}. The following result is immediate. 
 
 \begin{proposition} \begin{enumerate}
 \item The solutions of \eqref{eq2} cannot attain the value $r=0$ neither $r=1$. 
 \item The only constant solution of \eqref{eq2} is $r(s)=\frac12$, which it represents the circle of radius $r=\frac12$.
 \end{enumerate}
 \end{proposition}
 \begin{proof} The second statement is immediate. For the first one, if $r(s)=0$ at some point, then \eqref{eq2} gives $r'(s)=0$ and the curve $\gamma$ is not regular at $s$. Similarly, if $r(s)=1$ at some point, then \eqref{eq2} yields $r'(s)^2+1=0$, which it is not possible.
 \end{proof}

We now study of solutions of Eq. \eqref{eq2}. For this, we consider  initial conditions. Without loss of generality, we assume $0\in I$. Let 
\begin{equation}\label{ini}r(0)=r_0>0,\quad r'(0)=0.
\end{equation}
The condition $r'(0)=0$ implies that the solution $r(s)$ is symmetric.  

\begin{proposition} \label{pr-s}
Any solution of  \eqref{eq2}-\eqref{ini} is symmetric about   the vertical axis $s=0$.
\end{proposition}

\begin{proof}
 Define the function $\bar{r}(s)=r(-s)$. Then $\bar{r}$ satisfies \eqref{eq2}-\eqref{ini}, so by uniqueness, $\bar{r}(s)=r(-s)$ for all $s\in I$.  

\end{proof}

 Equation \eqref{eq2} is an autonomous differential equation. The behaviour of the solutions of this equation can be studied as an ODE system by introducing two new functions. As usually, let $u=r$ and $v=r'$.  Then  \eqref{eq2} is equivalent to
\begin{equation}\label{eqs}
\left(\begin{array}{l}u\\ v\end{array}\right)'= \left(\begin{array}{c}v\\ \dfrac{u(2u-1)}{u-1}+\dfrac{3u-2}{u(u-1)}v^2\end{array}\right)
\end{equation}
The phase plane $A=\{(u,v)\in\r^2:u\in (0,1)\cup (1,\infty), v\in\r\}$ is the space of solutions of \eqref{eqs}. Existence of uniqueness of the ODE system implies that two solutions of \eqref{eqs}, viewed as trajectories in $A$, cannot intersect. This provides a foliation of the set $A$.

The unique equilibrium point is $P=(\frac12,0)$, which corresponds with the constant solution $r(s)=1/2$. The linearization of the ODE system at $P$ gives 
$$\left(\begin{array}{ll}0&1\\ -2&0\end{array}\right).$$
This matrix has two purely imaginary distinct eigenvalues. So the linearized system has a  center at $P$. As illustrated in the phase plane portrait sketched in Fig. \ref{figphase}, all trajectories contained in the domain $(0,1)\times\r$ go around the equilibrium point $P$. The direction of rotation is obtained by picking a point, e.g., $(u,v)=(1,0)$. At this point we find the vector $(0,-2)$, which points in the clockwise direction. Hence, the trajectories look like ellipses with the clockwise direction on them. Since the trajectories are closed curves, this implies that the solutions are periodic functions. 

Another consequence of the phase portrait of Fig. \ref{figphase} is that all trajectories cross the $u$-axis. This implies that the value $v=0$ is   attained in any trajectory. Consequently, the initial condition $r'(0)=0$ in \eqref{ini} does not loose generality in the initial conditions. Furthermore, the trajectories contained in the subdomain $(0,1)\times\r$ meet the $u$-axis only at two points, whereas of the subdomain $(1,\infty)\times\r$ meet the $u$-axis only at one point.

 \begin{figure}[hbtp]
 \begin{center}
\includegraphics[width=.5\textwidth]{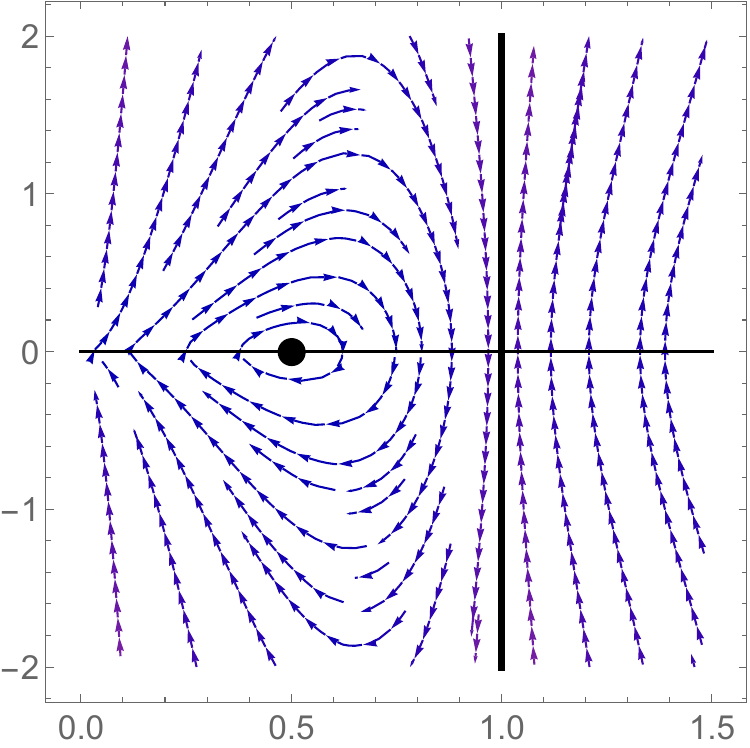} 
\end{center}
\caption{Phase portrait of the ODE system \eqref{eqs}. The equilibrium point is $P=(\frac12,0)$.  On the left side of the vertical line $u=1$ (black), the trajectories are closed curves around $P$. On the right side of $u=1$, the curves are graphs on the $v$-line, going the trajectories from $-\infty$ to $\infty$.  }
\label{figphase}
\end{figure}

We will obtain a first integration of \eqref{eq2}.

\begin{proposition}\label{pra} Let $r=r(s)$ be a solution of \eqref{eq2}-\eqref{ini}. Then the function $r(s)$ is solution of
\begin{equation}\label{fi}
\pm \int_{r_0}^r\frac{dr}{\sqrt{\pm(r^2(r-1)^2-r_0^2(r_0-1)^2})}\, dr=\frac{s}{r_0(r_0-1)}+c_1,
\end{equation}
where $c_1\in\r$ is an integration constant.
\end{proposition}

\begin{proof}
  Let introduce $w=r'$, where $w=w(r)$ and $p=w^2$. Then \eqref{eq2} becomes 
 $$r(r-1)p'+(4-6r)p+2r^2(1-2r). $$
 This is a linear first order equation. We solve with the Bernouilli method. So, let $p(r)=f(r)g(r)$. Then we can write the above equation as
 $$f(r)g'(r)+g(r)\left(f'+\frac{4-6r}{r(r-1)}\right)=\frac{4r^2-2r}{r-1}.$$
 Vanishing the parenthesis, we obtain $f(r)=r^4(r-1)^2$. Replacing in the above equation, we have
 $$g'(r)=\frac{2(2r-1)}{r^3(r -1)^3}.$$
 We solve $g$ by separation of variables, obtaining
 $$g(r)=c-\frac{1}{r^2(r-1)^2},\quad c\in\r.$$
 Thus $p=r^2(c(r-1)^2-1)$. Coming back, we have 
 $$r'^2=r^2(cr^2(r-1)^2-1).$$
 At $s=0$, we get the constant $c$, namely, 
 $$c=\frac{1}{r_0^2(r_0-1)^2}.$$
 Then
  \begin{equation}\label{eqa}
  r'^2r_0^2(r_0-1)^2=r^2(r^2(r-1)^2-r_0^2(r_0-1)^2).
  \end{equation}
From this equation, we obtain $r'$ and integrating by separable variables, we get \eqref{fi}. 
\end{proof}

With the aid of Prop. \ref{pra} together the phase portrait (Fig. \ref{figphase}) we will give a complete description of the catenaries. We distinguish if $r_0$ is less or bigger than $1$ in the initial condition \eqref{ini}.

 \begin{theorem}[Case $r_0<1$]\label{t1}
 Let $r=r(s;r_0)$ by a solution of \eqref{eq2}-\eqref{ini}. Suppose $r_0<1$. Then $r$ is the constant function $r(s)=\frac12$ or $r(s)$ satisfies the following properties:
 \begin{enumerate}
 \item The function $r(s)$ is periodic. Let $T>0$  denote its period.  In each interval $[0,T)$ the function $r(s)$ has exactly one minimum, which is less than $1/2$, and one maximum, which is bigger than $1/2$.
 \item We have $ r(s;1-r_0)=r(s+T;r_0)$.
 \item If $r_0\to 0$ (resp. $r_0\nearrow 1$), then $\gamma$ converges to a double covering of the segment $\{0\}\times (-1,1)$ (resp.   of   $ (-1,1)\times \{0\}$).

 \end{enumerate}
 \end{theorem}
 
 \begin{proof} We know that if $r_0=1/2$, then $r(s)=1/2$ for all $s\in I$ by a direct computation. Suppose now $r_0\not=1/2$.
 \begin{enumerate}
 \item We know that the point $P$ is a center of the ODE system \eqref{eqs}. Thus  all trajectories of the phase plane  that lie in the halfplane $(0,1)\times\r$ go around $P$. This proves that all solutions with $r_0<1$ are periodic. Moreover, all trajectories intersect  the $u$-axis only at two points. This proves that the solution $r=r(s)$ has a minimum (less than $1/2$) and a maximum (bigger than $1/2$) in each interval $[0,T)$ of its domain.  
 \item Consider the interval $[0,T)$ of the domain of $r=r(s)$. We know that $r'(0)=0$ and that there is a unique point $s_1\in (0,T)$ where $r'$ vanishes. From Eq. \eqref{eq2}, we have
 $$r'(s_1)(r(s_1)-1)=r_0(r_0-1).$$
 Then $r(s_1)=r_0$ or $r'(s_1)=1-r_0$. Since the first case is not possible by the phase portrait, then $r'(s_1)=1-r_0$. Moreover, from \eqref{eq2}, we have 
 $$r''(s_1)=\frac{2r_0-1}{r_0}.$$
 This proves that if $s=0$ is a minimum (resp. maximum) of $r(s)$, then $s=s_1$ is a maximum (resp. minimum). 
  
 Define the function $\bar{r}(s)=r(s+T)=r(s+T;r_0)$. Then it is immediate that $\bar{r}(s)$ satisfies \eqref{eq2} and the initial conditions at $s=0$ are $\bar{r}(0)=1-r_0$ and $\bar{r}'(0)=0$. This proves that $r(s+T;r_0)=r(s;1-r_0)$ by uniqueness of solutions of \eqref{eq2}-\eqref{ini}.  
 \item The last statement is consequence of the phase portrait. Indeed, if $r_0\to 0$, then the trajectory is asymptotic to the $v$-axis. This implies that $r\to 0$ hence that $\gamma$ is asymptotic to the segment $\{0\}\times (-1,1)$. If $r_0\to 1$ the argument is similar. We can also use the above property (2).
\end{enumerate}
 \end{proof}
 
  In Fig. \ref{fig2} we show two catenaries with $r_0=1/4$ and $r_0=3/4$. In particular, the corresponding functions $r(s;\frac14)$ and $r(s;\frac34)$ coincide after a translation on the parameter $s$ according (2) of Thm. \ref{t1}: see Fig. \ref{fig3}.

   \begin{figure}[hbtp]
 \begin{center}
\includegraphics[width=.3\textwidth]{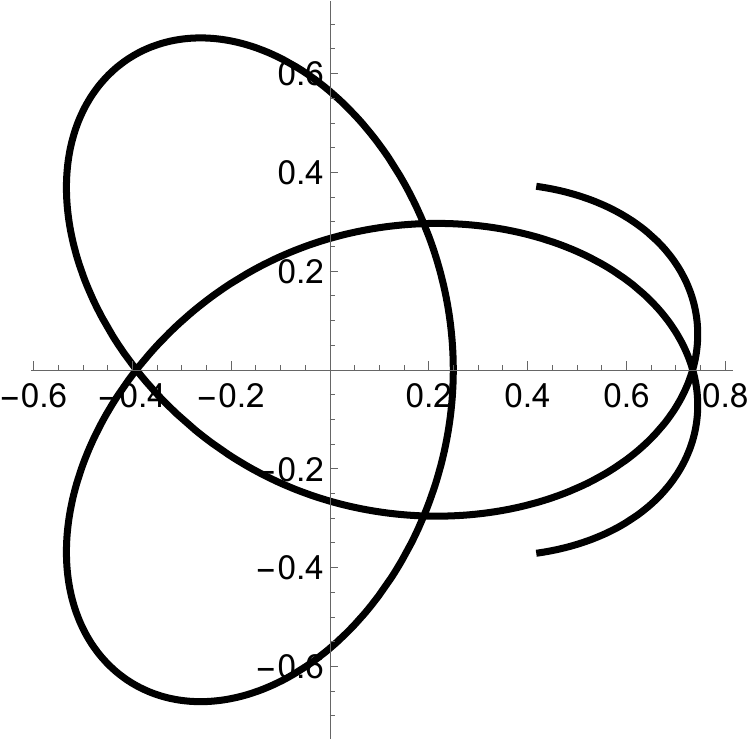}  \quad \includegraphics[width=.3\textwidth]{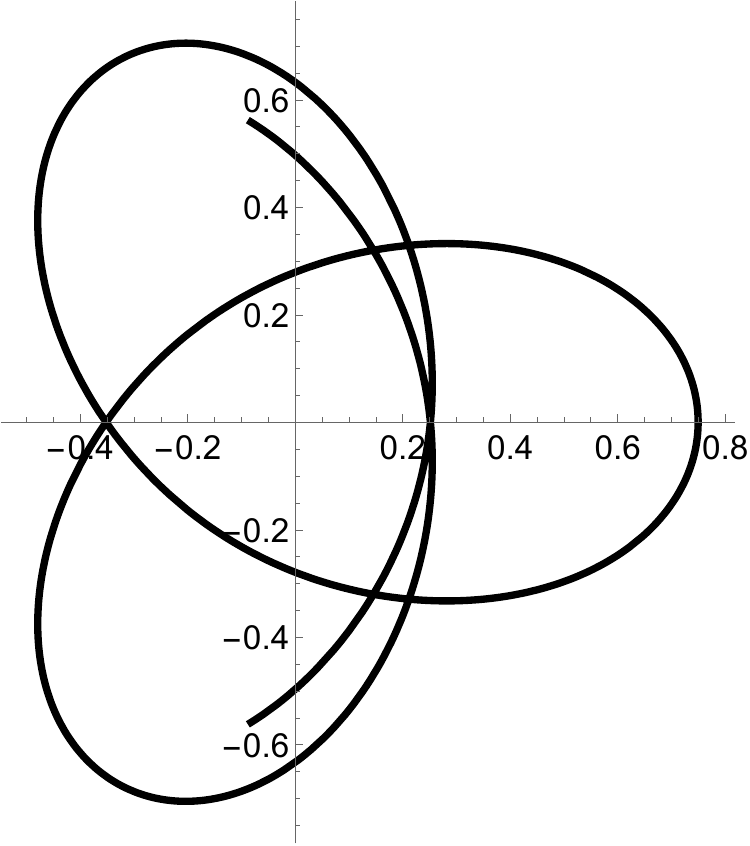}   
\end{center}
\caption{Solutions of \eqref{eq2}-\eqref{ini}. Here  $r(0)=1/4$ (left) and $r(0)=3/4$ (right). }
\label{fig2}
\end{figure}

   \begin{figure}[hbtp]
 \begin{center}
\includegraphics[width=.3\textwidth]{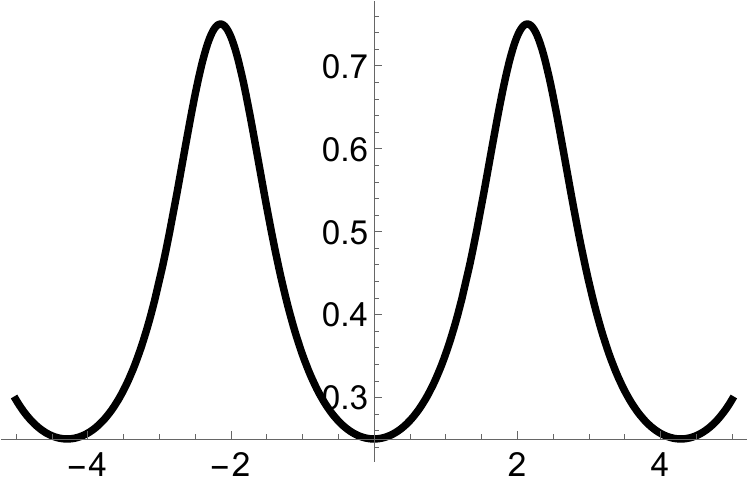}  \quad \includegraphics[width=.3\textwidth]{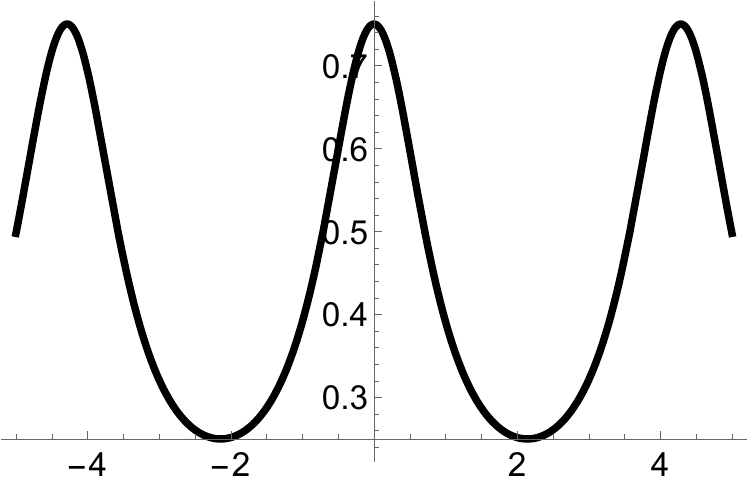}   
\end{center}
\caption{Graphics of the function $r=r(s)$ of \eqref{eq2}-\eqref{ini}. Here  $r(0)=1/4$ (left) and $r(0)=3/4$ (right). }
\label{fig3}
\end{figure}

An interesting problem  is on the existence of  catenaries that are closed curves. The fact that the function $r(s)$ is periodic and, consequently, also its curvature $\kappa(s)$, is a necessary condition for the curve to be closed, but not sufficient. The problem when a periodic function is the curvature of a closed planar curve is difficult to solve in all its generality: see an interesting discussion and some results in \cite{ar}.

We now consider the case $r_0>1$. We will se that the catenaries are radial graphs on sub-arcs of $\s^1$, that is, each ray from the origin intersects the catenary at one point at most. 
 
 \begin{theorem}[Case $r_0>1$]\label{t2}
 Let $r=r(s)$ a solution of \eqref{eq2}-\eqref{ini}. Suppose $r_0>1$. Then $r$  satisfies the following properties:
 \begin{enumerate}
 \item The function $r(s)$ has only a critical point, namely, $s=0$, which it is a minimum.
 \item The function $r$ is convex.
  \item The domain of the function $r(s)$  is a bounded interval $(-s_1,s_1)$, $s_1<\frac{\pi}{2}$. The catenary $\gamma$ is a radial graph   on the sub-arc $(-s_1,s_1)$ of $\s^1_+:=\{(x,y)\in\s^1:x>0\}$. Moreover, $r$ has  two vertical asymptotes at the points $s=\pm s_1$, that is, $\gamma$ is asymptotic   to the two rays from the origin at angles $\pm s_1$.

 \end{enumerate}
 \end{theorem}

\begin{proof} 
\begin{enumerate}
\item We know by the phase portrait that the trajectories in the subdomain $(1,\infty)\times\r$ are curves which are graphs on the $v$-axis. We know by Prop. \ref{pr-s} that   $r$ is symmetric about the line $s=0$. Moreover $s=0$ is a critical point, in fact, a minimum because $r''(0)=\frac{r_0(2r_0-1)}{r_0-1}>0$.  
 \item  For the convexity, from \eqref{eq2} we obtain
  $$r''=\frac{r'^2(3r-2)+r^2(2r-2)}{r(r-1)}.$$
  The denominator and numerator are positive because $r>1$. This proves $r''>0$ everywhere, as desired. \item We prove that  $r(s)$ is defined in a bounded interval $(-s_1,s_1)$. Following \eqref{fi},  we can be assume $c_1=0$ by translating  the parameter $s$ of $\gamma$. 
 Since the function $r$ is increasing for $s>0$, then we have 
 $$\frac{s}{r_0(r_0-1)}=\int_{r_0}^r\frac{dr}{r\sqrt{r^2(r-1)^2-r_0^2(r_0-1)^2}}.$$
 We know that $r-1\geq r_0-1$ because $r(s)>r_0$ for all $s$, $s\not=0$. Therefore
\begin{align*}
\frac{s}{r_0(r_0-1)}&\leq\frac{1}{r_0-1}\int_{r_0}^r\frac{dr}{r \sqrt{r^2 -r_0^2 }}=\frac{1}{r_0(r_0-1)}\tan^{-1}(\frac{\sqrt{r^2-r_0^2}}{r_0})\\
&\leq \frac{\pi}{2r_0(r_0-1)}.
\end{align*}
 Thus $s<\frac{\pi}{2}$. This proves that the domain of the function $r(s)$ is included in the interval $(-\frac{\pi}{2},\frac{\pi}{2})$.   Since the trajectories in the phase portrait goes to $\pm\infty$, this means that the limit of $r'$ is $\infty$. Thus $r$ goes to $\infty$ at the ends $\pm s_1$ of its domain.  
\end{enumerate}
\end{proof}

In Fig. \ref{fig4} we show some pictures of catenaries on $\s^1$ for several values $r_0>1$.   

 \begin{figure}[hbtp]
 \begin{center}
\includegraphics[width=.15\textwidth]{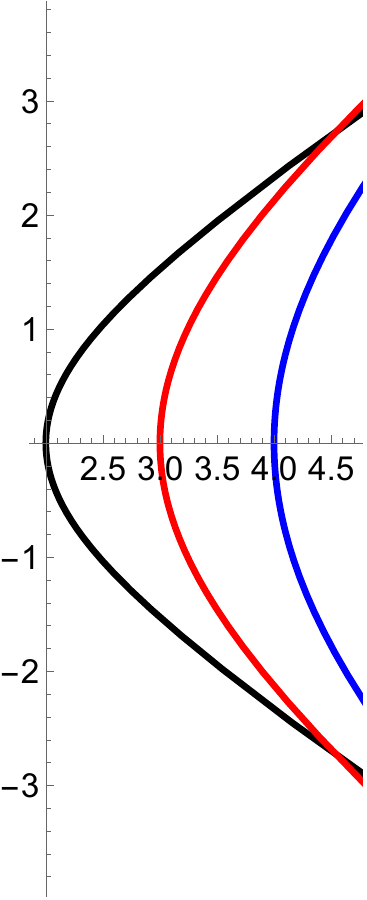}  
\end{center}
\caption{Solutions of \eqref{eq2}. Here  $r(0)=2$ (black), $r(0)=3$ (red) and $r(0)=4$ (blue). }
\label{fig4}
\end{figure}

\section{Generalization to other energies}\label{s3}

We generalize the hanging chain problem with respect to a circle   by considering powers of the distance to $\s^1$. Let $\alpha\in\r$. For a planar curve $\gamma$ given in polar coordinates $r=r(s)$, define the functional $E_\alpha$ as 
\begin{equation}\label{e2}
E_\alpha[\gamma]= \int_{s_1}^{s_2} |r-1|^\alpha \sqrt{r^2+r'^2}\, ds.
\end{equation}
If $\alpha=0$ then $E_0$  represents the length of the curve and it is known that the extremals are the geodesics of $\r^2$, that is, straight-lines. This case will be discarded. 

 The purpose in this section is, again, to give  a geometric description of the extremals of $E_\alpha$. Since in many of the following results, the proofs are analogous as in Sect. \ref{s2}, we will omit the details if necessary.  Extremals of $E_\alpha$ are calculated   as the case $\alpha=1$. In the next result we find the Euler-Lagrange equation as well as the analogous of Prop. \ref{pr21}.

\begin{proposition} A planar curve $r=r(s)$ is an extremal  of $E_\alpha$ if and only if 
\begin{equation}\label{sol2}
r(r-1)r''=\left((\alpha+2)r-2\right)r'^2+\left((\alpha+1)r-1\right)r^2.
\end{equation}
This equation is equivalent to 
\[
\kappa=\alpha\frac{\cos\varphi}{r-1},
\]
where $\varphi$ is the angle between the unit normal  $N$ of $\gamma$ with the opposite of the vector $\gamma$.
\end{proposition}

As in the case $\alpha=1$, the solutions of \eqref{sol2} do not cross the origin of $\r^2$ neither the circle $\s^1$. 

 If $u=r$ and $v=r'$, then \eqref{sol2} is equivalent to the ODE system
\begin{equation}\label{eqs2}
\left(\begin{array}{l}u\\ v\end{array}\right)'= \left(\begin{array}{c}v\\ \dfrac{u((\alpha+1)u-1)}{u-1}+ \dfrac{(\alpha+2)u-2}{u(u-1)}v^2\end{array}\right)
\end{equation}
The phase plane is $A=\{(u,v)\in\r^2:u\in (0,1)\cup (1,\infty), v\in\r\}$. Equilibrium points only appear when $\alpha>-1$, begin this point  $P_\alpha=(\frac{1}{1+\alpha},0)$ if $\alpha>-1$; otherwise, there are not. The equilibrium point $P_\alpha$ corresponds with the constant solution $r(s)=\frac{1}{1+\alpha}$, that is, a circle centered at the origin of   radius $\frac{1}{1+\alpha}$. The linearization of \eqref{eqs2} at $P_\alpha$ is  
$$\left(\begin{array}{ll}0&1\\ -\frac{\alpha+1}{\alpha}&0\end{array}\right).$$
The eigenvalues are two purely imaginary distinct numbers if $\alpha>0$ or two real distinct number if $\alpha\in (-1,0)$. Thus we have (see Fig. \ref{fig5}):
\begin{enumerate}
\item If $\alpha>0$, then $P_\alpha$ is a center.
\item If $\alpha\in (-1,0)$, then $P_\alpha$ is an unstable saddle point.
\end{enumerate}
   \begin{figure}[hbtp]
 \begin{center}
\includegraphics[width=.42\textwidth]{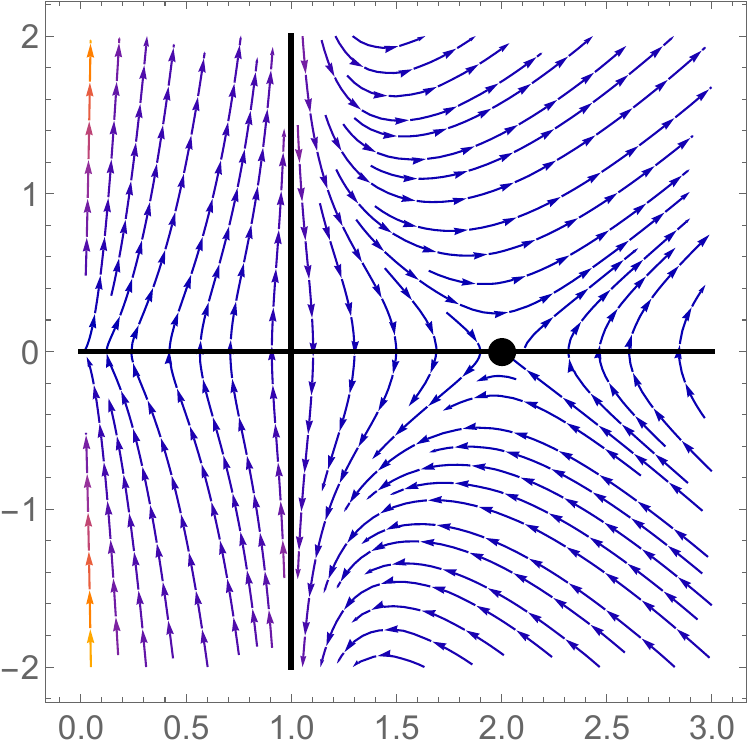}  \qquad \includegraphics[width=.42\textwidth]{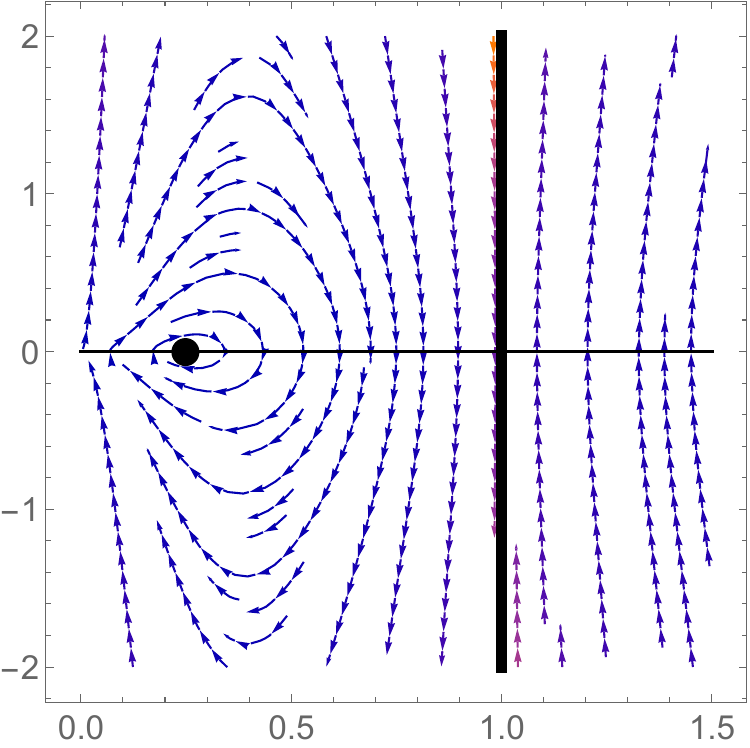}   
\end{center}
\caption{Phase portrait of the ODE system \eqref{eqs2}: case $\alpha=-\frac12$ (left) and $\alpha=3$ (right).   }
\label{fig5}
\end{figure}

Proposition \ref{pra}  cannot be easily generalized   because of the arbitrariness of $\alpha$. When $\alpha$ is an integer,  it is possible to get an expression similar to \eqref{fi}.  

\begin{proposition} 
Let $r=r(s)$ be a solution of \eqref{sol2}. If $\alpha\in\mathbb {Z}$, then there is a first integral of Eq. \eqref{sol2}. In the particular case $\alpha\in\mathbb{Z}^+$, a first integral of \eqref{sol2}-\eqref{ini} is   
\begin{equation}\label{eqa2}
  r'^2r_0^2(r_0-1)^{2\alpha}=r^2(r^2(r-1)^{2\alpha}-r_0^2(r_0-1)^{2\alpha}).
  \end{equation}  
\end{proposition} 

\begin{proof}
 Repeating the proof of Prop. \ref{pra}, we have $v(r)=f(r)g(r)$, where 
$$f(r)g'(r)+g(r)\left(f'-2\frac{(\alpha+2)r-2}{r(r-1)} \right) =2\frac{r((\alpha+1)r-1)}{r-1}.$$
The function $f$ is obtained by vanishing the parenthesis in the left hand-side of this equation. This yields $f(r)=r^4(r-1)^{2\alpha}$. Once we know the function $f$, the function $g$ satisfies the ODE
\begin{equation}\label{gg}
g'(r)=2\frac{(\alpha+1)r-1}{r^3(r-1)^{2\alpha+1}}.
\end{equation}
If $\alpha\in\mathbb{Z}$, then we can apply standard methods of integration because the denominator in the right hand-side of \eqref{gg} is a quotient of polynomials and the denominator  can be decomposed in factors with integer coefficients. This proves the existence of a first integral. In the particular case   $\alpha\in\mathbb{Z}^+$,      the solution of \eqref{gg} is  
$$g(r)=c-\frac{1}{r^2(r-1)^{2\alpha}}.$$
After imposing the initial conditions,  we obtain   \eqref{eqa2}
  \end{proof}
  
  \begin{remark} If $\alpha\in\mathbb{Z}^-$, it is possible to solve \eqref{gg} being difficult to find a general formula. In such a case, we have 
  $$g(r)=c-\frac{1}{r^2}-\frac{2\alpha}{r}+P(r),$$
  where $P(r)$ is a polynomial of degree $-2\alpha$. Some particular examples are the following
  \begin{enumerate}
  \item If $\alpha=-1$, then $P(r)=0$.
  \item If $\alpha=-2$, then $P(r)=4r-r^2$.
  \item If $\alpha=-3$, then $ P(r)=20r-15r^2+6r^3-r^4$.
  \end{enumerate}
  \end{remark}

  We begin to analyze the solutions of \eqref{sol2} depending on the value of $\alpha$.   The first case to consider is $\alpha>0$. By the phase portrait, this case is similar to $\alpha=1$ and the proofs are similar. If $r_0>1$, we need a first integral of Eq. \eqref{sol2}. This integral was obtained when $\alpha\in\mathbb{Z}^+$. For this reason, in the next theorem we will assume under that situation that $\alpha$ is an integer. However, numerical computations show that the result is still valid for any $\alpha>0$,  regardless of whether or not $\alpha$ is integer. 
  
   \begin{theorem} \label{t3}
 Let $\alpha>0$. Let $r=r(s)$ by a solution of \eqref{sol2}-\eqref{ini}. The only  constant solution is $r(s)=\frac{1}{1+\alpha}$ and the catenary is a circle of radius $\frac{1}{1+\alpha}$. In other case, the function $r(s)$ satisfies the following properties:
 \begin{enumerate}
 \item Suppose $r_0<1$ with $r_0\not=\frac{1}{1+\alpha}$. Then the function $r(s)$ is periodic. If $T>0$  denotes its period, then in each interval $[0,T)$ there is exactly one minimum, which is less than $\frac{1}{1+\alpha}$, and one maximum, which is bigger than $\frac{1}{1+\alpha}$.
  \item Suppose $r_0>1$.    Then the function $r(s)$   is a convex symmetric  graph with only a critical point, namely, $s=0$, which it is a minimum. If $\alpha\in\mathbb{Z}^+$, then the domain of the function $r(s)$  is a bounded interval $(-s_1,s_1)$ with $s_1\leq\frac{\pi}{2}$. The curve $\gamma$ is a radial graph on the sub-arc $(-s_1,s_1)$  and asymptotic to the   two     rays  $s=\pm s_1$. 
 
 \end{enumerate}
 \end{theorem}
 
  \begin{proof} The only detail is when  $r_0>1$ and $\alpha$ is an integer. Then \eqref{eqa2} gives
\begin{align*}
\frac{s}{r_0(r_0-1)^\alpha}&\leq\frac{1}{(r_0-1)^\alpha}\int_{r_0}^r\frac{dr}{r \sqrt{r^2 -r_0^2 }} \leq \frac{\pi}{2r_0(r_0-1)^\alpha}.
\end{align*}
 Thus $s<\frac{\pi}{2}$ and proving the result.
 \end{proof}
 
 We now consider the case $\alpha\in (-1,0)$. See Fig. \ref{fig6}.
 
   \begin{theorem} \label{t4}
 Let $\alpha\in (-1,0)$. Let $r=r(s)$ by a solution of \eqref{sol2}-\eqref{ini}. The only  constant function is $r(s)=\frac{1}{1+\alpha}$ and the catenary is a circle of radius $\frac{1}{1+\alpha}$. In other case, the catenary is defined in a bounded interval $(-s_1,s_1)$ with $\lim_{s\to\pm s_1}r'(s)=\infty$. Moreover, 
 \begin{enumerate}
 \item If   $r_0<1$ (resp. $ 1<r_0<\frac{1}{1+\alpha}$), the function $r(s)$ is convex (resp. concave).  The catenary  intersects orthogonally the circle $\s^1$ at two points.
 \item If $r_0>\frac{1}{1+\alpha}$, then the function $r(s)$ is convex with with 
 $$\lim_{s\to\pm s_1}r(s)=\lim_{s\to\pm s_1}r'(s)=\infty.$$ 
 \end{enumerate}
 \end{theorem}
 
 \begin{proof}
Notice that $r''(0)=\frac{((\alpha+1)r_0-1)^2r_0^2}{r_0(r_0-1)}$.   By the phase portrait (Fig. \ref{fig5}, left), the trajectories only meet once that $u$-axis, which corresponds with $s=0$. By Eq. \eqref{sol2}, we have 
$$r''=\frac{((\alpha+2)r-2)r'^2+((\alpha+1)r-1)r^2}{r(r-1)}.$$
Then $r$ is convex  if   $r_0\in (0,1)\cup (\frac{1}{1+\alpha},\infty)$ and concave if   $r_0\in (1,\frac{1}{1+\alpha})$. 

Consequently, if $r_0\in (0,1)$ and because the function is convex, then  $r$ attains the value $1$. To be precise, the maximal domain of $r$ is a bounded interval   $(-s_1,s_1)$ where  $\lim_{s\to\pm s_1}r(s)=1$. Moreover, because $r$ does not attain the value $1$, necessarily $r'$ is $\infty$ at $\pm s_1$.  Similarly, if $r_0\in (1,\frac{1}{1+\alpha})$, and by concavity of $r$, the function $r$ is defined again in a interval $(-s_1,s_1)$ with $\lim_{s\to\pm s_1}r(s)=1$ and $\lim_{s\to\pm s_1}r'(s)=\infty $. To calculate the angle of intersection with $\s^1$, we use the expression of $N$ in \eqref{nn}. Then  $\lim_{s\to\pm s_1}\langle N,\gamma\rangle=0$, proving that the intersection of the catenary with $\s^1$ must be orthogonal. 
 
  Suppose $r_0>\frac{1}{1+\alpha}$. Then we know that $r$ is convex and $r$ is increasing if $s>0$. 
 
 \end{proof}

    \begin{figure}[hbtp]
 \begin{center}
\includegraphics[width=.35\textwidth]{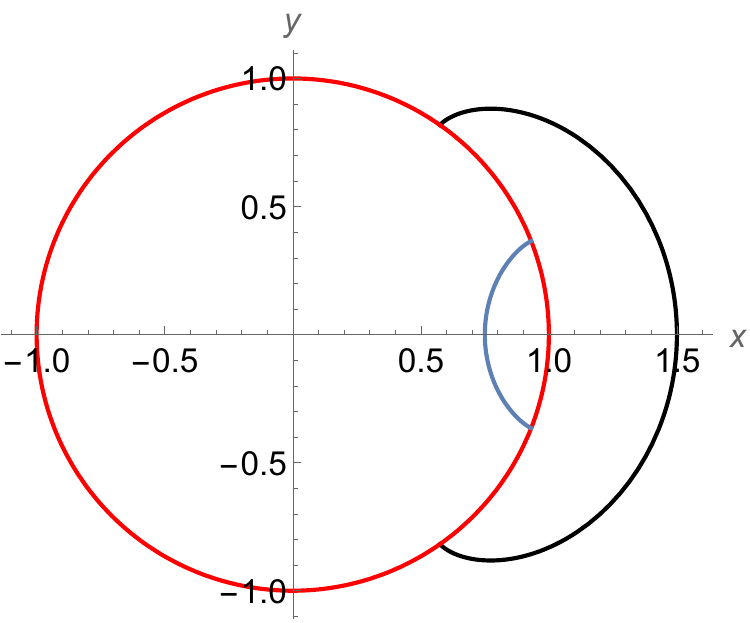}  \quad \includegraphics[width=.5\textwidth]{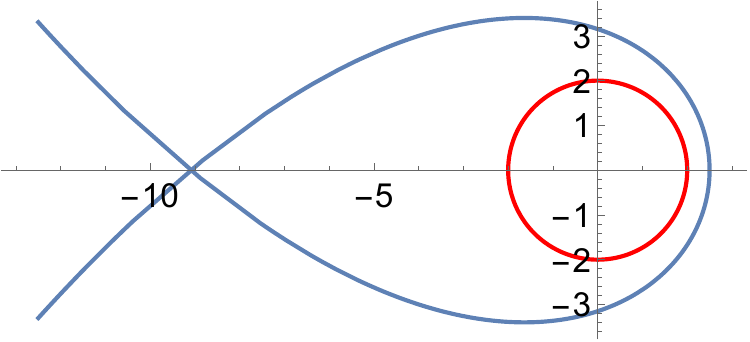}   
\end{center}
\caption{Case $\alpha\in (-1,0)$. Here $\alpha=-\frac12$. Solutions of \eqref{sol2}-\eqref{ini}.  Left:   $r_0=.75$, $r_0=1.5$ and the circle $\s^1$ (red). Right:     $r_0=2.5$ and the constant solution $r=\frac{1}{1+\alpha}=2$ (red).}
\label{fig6}
\end{figure}

\begin{remark} If $\alpha\in (-1,0)$, the behaviour of the catenary is difficult to describe completely. For example, when $r_0<\frac{1}{1+\alpha}$ the maximal interval can be big so the catenary can turn around the origin before to intersect the circle $\s^1$. See Fig. \ref{fig7}.
\end{remark}

    \begin{figure}[hbtp]
 \begin{center}
\includegraphics[width=.35\textwidth]{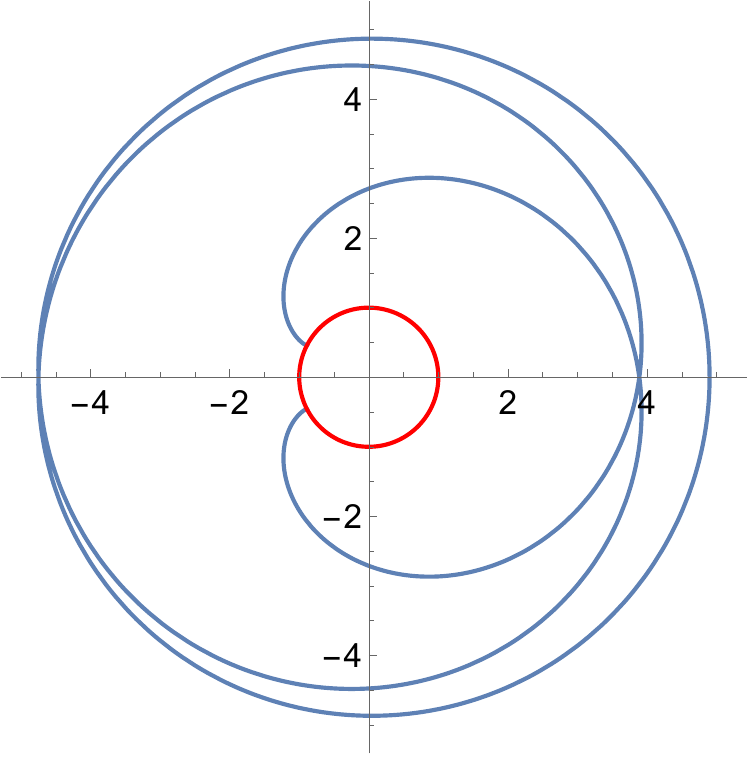}  \quad \includegraphics[width=.4\textwidth]{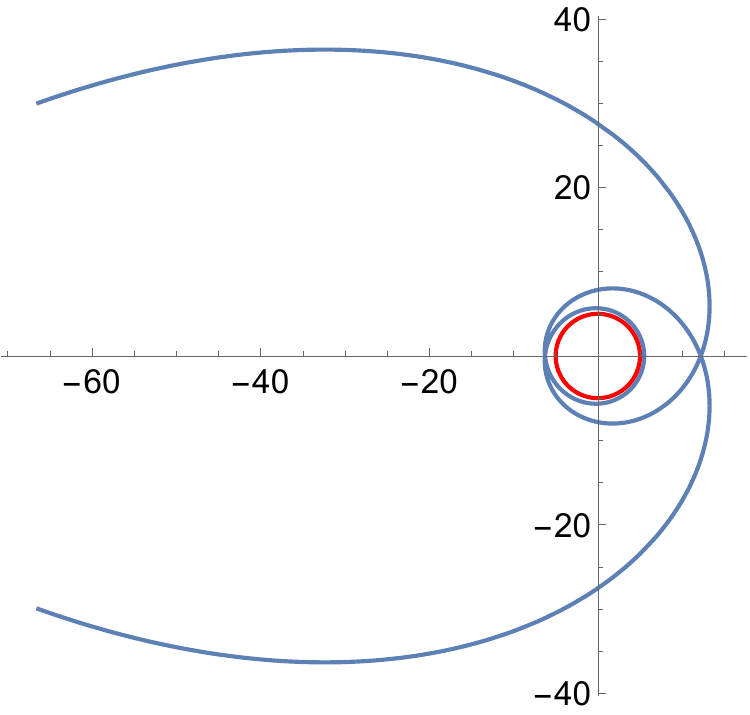}   
\end{center}
\caption{Case $\alpha\in (-1,0)$. Here $\alpha=-0.8$, so $\frac{1}{1+\alpha}=5$. Left:   $r_0=4.9$ with the circle $\s^1$ (red). Right: $r_0=5.5$ with the circle $r=5$ (red). }
\label{fig7}
\end{figure}

Finally consider the case $\alpha\leq -1$. See Fig. \ref{fig8}.

\begin{theorem} \label{t6}
 Let $\alpha\leq -1$. Let $r=r(s)$ by a solution of \eqref{sol2}-\eqref{ini}. Then the function $r(s)$ is defined in a bounded interval $(-s_1,s_1)$ being symmetric with respect to the $r$-axis. Moreover, the  catenary   intersects orthogonally $\s^1$ at two points being $r$ a convex (resp. concave) function if $r_0<1$ (resp. $r_0>1$). 
 \end{theorem}
 
 \begin{proof} This situation is similar to that of $\alpha\in (-1,0)$ when $r_0<\frac{1}{1+\alpha}$. Recall that now there are no equilibrium points because $1+\alpha\leq 0$. On the other hand, and by the phase portrait (Fig. \ref{fig8}, left), each trajectory of the phase plane intersects the $u$-axis at exactly one point. This proves that $r(s)$ attains only a critical point, and this occurs precisely at $s=0$. The function $r''$ has constant sign being $r''>0$ if $r_0<1$ and $r''<0$ if $r_0>1$. If $r_0<1$, then $r$ is convex. Since $r$ is an increasing and convex function, then $r$ attains the value $r=1$ at the end of its domain. This proves that this domain is a bounded interval $(-s_1,s_1)$. Similarly, if $r_0>1$, the function $r$ is concave and decreasing if $s>0$. This shows that $r$ also attains the value $r=1$. 

 \end{proof}
 
    \begin{figure}[hbtp]
 \begin{center}
\includegraphics[width=.4\textwidth]{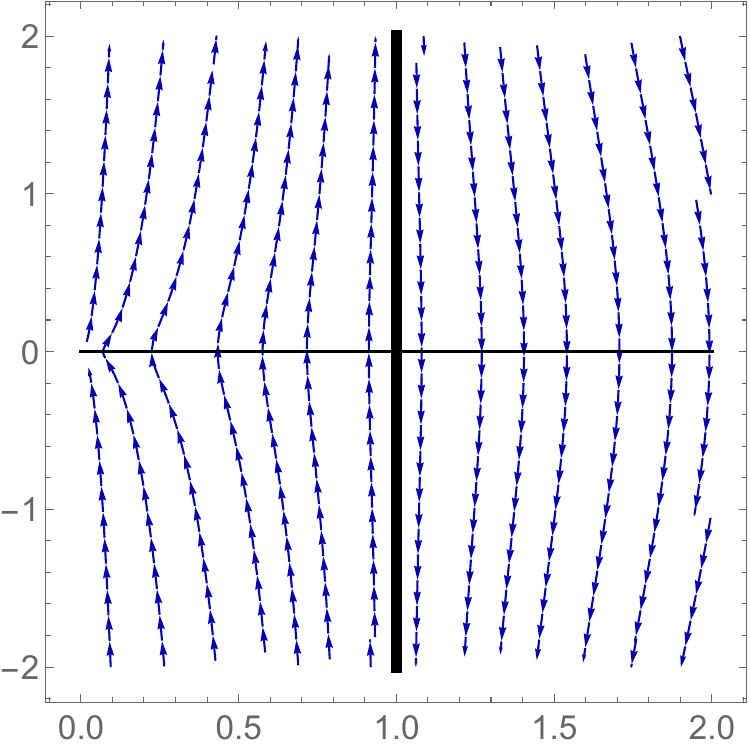}  \quad \includegraphics[width=.5\textwidth]{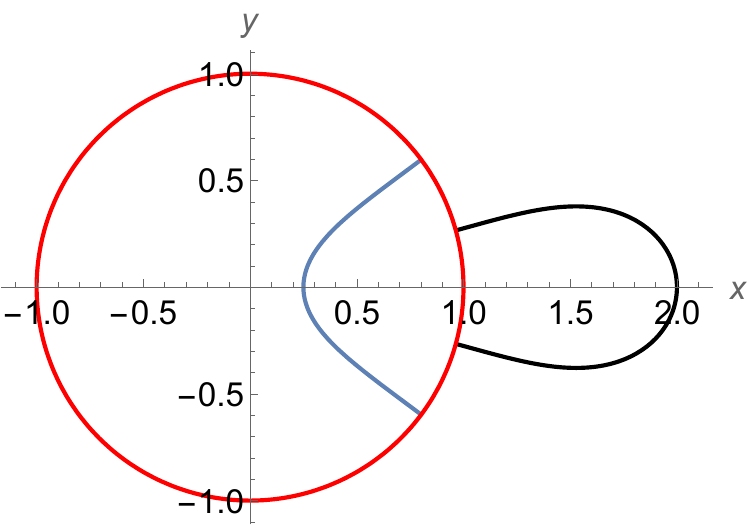}   
\end{center}
\caption{Case $\alpha\leq 1$. The phase portrait (left) and two solutions of   \eqref{sol2}-\eqref{ini}  when $\alpha=-3$. Here  $r_0=.25$ and $r_0=2$. In both cases, we show the circle $\s^1$.}
\label{fig8}
\end{figure}

Numerical evidences show that the maximal domain $(-s_1,s_1)$ of $r(s)$ is included in $(-\frac{\pi}{2},\frac{\pi}{2})$. Consequently, the catenary is a radial graph on a sub-arc of $\s^1_{+}$: see Fig. \ref{fig8}, right.

The case $\alpha=-2$ is special because inversions with respect $\s^1$ preserve the hanging chain problem with respect to $\s^1$. The following result is immediate. See Fig. \ref{fig9}.

\begin{corollary} Suppose $\alpha=-2$. If $r=r(s)$ is a solution of \eqref{eq2}, then $\rho=\rho(s)=\frac{1}{r(s)}$ is also a solution of \eqref{eq2} for $\alpha=-2$. In particular, we have $r(s;r_0)=\frac{1}{r(s;\frac{1}{r_0})}$.
\end{corollary}
 
  \begin{figure}[hbtp]
 \begin{center}
\includegraphics[width=.5\textwidth]{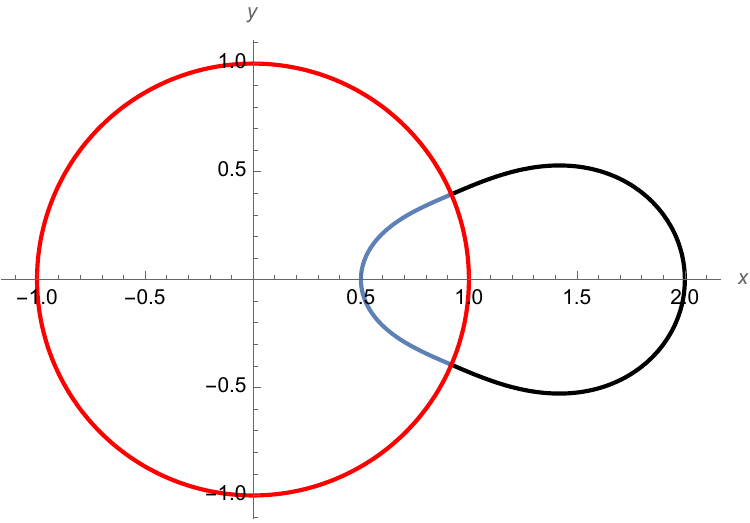}     
\end{center}
\caption{Case $\alpha=-2$. The solutions of     \eqref{sol2}-\eqref{ini}  for initial conditions $r(0)=\frac12$ and  $r(0)=2$.  }
\label{fig9}
\end{figure}
\section*{Acknowledgements}  

The author thanks to Prof. Maciej Czarnecki   for proposing the problem and   acknowledges his hospitality in the University of Lodz  during the preparation of this paper.  This research has been partially supported by MINECO/MICINN/FEDER grant no. PID2020-117868GB-I00, and by the ``Mar\'{\i}a de Maeztu'' Excellence Unit IMAG, reference CEX2020-001105- M, funded by MCINN/AEI/10.13039/501100011033/ CEX2020-001105-M.

\section*{Declarations}

{\bf Data Availability.} There is no additional data and materials.\\
{\bf EthicalApproval.} Not applicable.\\
{\bf Competing interests.} There is no competing interests.\\
{\bf  Conflicts of interest.} The authors declares no conflict of interest. 
 
\end{document}